\documentclass[12pt,widepage]{article}
\usepackage[cp1251]{inputenc}
\usepackage[english]{babel}
\usepackage{euscript}
\usepackage{amsthm, amsmath, amssymb}
\usepackage{hyperref}
\newtheorem{thm}{Theorem}
\newtheorem{lemma}{Lemma}
\newtheorem{sled}{Corollary}
\newtheorem{utv}{Claim}
\theoremstyle{definition}
\newtheorem{opr}{Definition}
\theoremstyle{remark}

\DeclareMathOperator\LIP{LIP}
\DeclareMathOperator\KR{KR}
\DeclareMathOperator\sign{sign}
\DeclareMathOperator\ROOT{Root}
\def\odin{\mathbf{1}}
\let\leq\leqslant
\let\geq\geqslant
\title{\textbf{Combinatorics of the Lipschitz polytope}}

\author{J.~Gordon$^b$, F.~Petrov$^{a,b}$}
\date{27.01.16}

\begin{document}
\maketitle

\begin{abstract}
Let  $\rho$ be a metric on the set
$X=\{1,2,\dots,n+1\}$.
Consider the $n$-dimensional
polytope of functions $f:X\rightarrow \mathbb{R}$, which satisfy
the conditions $f(n+1)=0$, $|f(x)-f(y)|\leq \rho(x,y)$. 
The question on classifying metrics depending on the combinatorics of this polytope have been recently posed by A.~M.~Vershik \cite{V}.
We prove that for any ``generic'' metric
the number of $(n-m)$-dimensional faces, $0\leq m\leq n$, equals
 $\binom{n+m}{m,m,n-m}=(n+m)!/m!m!(n-m)!$. This fact
is intimately related to regular triangulations
of the root polytope (convex hull of the roots of 
$A_n$ root system). Also we get two-sided estimates 
for the logarithm of the number of Vershik classes of metrics:
$n^3\log n$ from above and 
$n^2$ from below.
\end{abstract}
\section{Introduction}

{\let\thefootnote\relax\footnote{
a) St. Petersburg Department of
V.A. Steklov Institute of Mathematics of
the Russian Academy of Sciences;
b) St. Petersburg State University.
E-mails: fedyapetrov@gmail.com, joseph-gordon@yandex.ru.
Supported by Russian Scientific Foundation grant 14-11-00581.
}

Let \((X,\rho)\) be a finite metric space with \(|X|=n+1\). 
A metric $\rho$ will be called \emph{strict}, if $\rho(x,z)<\rho(x,y)+\rho(y,z)$
for $y\in X\setminus \{x,z\}$.
On the space of functions from $X$
\[
\tilde{F}:=\left\{ f:X\rightarrow \mathbb{R}\right\}
\]
we define a map, which maps every function to its Lipcshitz constant:
\[
\|f\|:=\max_{x,y\in X}{\frac{f(y)-f(x)}{\rho(x,y)}}
\]
It is a seminorm, becoming a norm if all functions differing by 
a constant are identified --- or,
equivalently, if the value is fixed at one point:
\[
F:=\left\{f\in\tilde{F} \mid f(x_0)=0\right\}, x_0\in X.
\]
The closed unit ball of the norm $\|\cdot\|$ on the space $F$ is a convex $n$-dimensional polytope,
which will be denoted by $\LIP(X)$. We will interpret the dual normed space $F^*$ as a space of
signed measures $\mu$ on the metric space $(X,\rho)$ with the total measure of 0, the pairing $\langle\mu,f\rangle$ of a function $f$ and a 
signed measure $\mu$ is $\int fd\mu$ (the value doesn't change 
when we add a constant to $f$, thus it is well defined.) The delta-measure at a point $x\in X$ 
will be denoted by $\delta_x$, and then the signed measure $\mu \in F^*$ 
has the following form: 
\begin{align*}
\mu=\sum_{x\in X} c_x\delta_x,\, \sum c_x=0;\\ \langle\mu,f\rangle=
\sum_{x\in X} c_xf(x),\,\|\mu\|=\max_{f\in \LIP(X)} \langle\mu,f\rangle.
\end{align*}
This norm of a signed measure $\mu\in F^*$ is 
called \emph{Kantorovich--Rubinstein norm}. It is equal to the Kantorovich optimal
transportation  distance
between the measures $\mu_+$ and $\mu_-$ coming 
from the Hahn decomposition of $\mu=\mu_+-\mu_-$.

The convex hull of the set of points of the form 
$e_{x,y}:=\frac{\delta(x)-\delta(y)}{\rho(x,y)},x,y\in X$ serves as a dual polytope to $\LIP(X)$, i.e.,
the
unit ball $\KR(X)$ of the Kantorovich--Rubinstein space $F^*$: indeed, the norm of a function
$f$, by definition, satisfies 
$$
\|f\|=\sup_{x,y\in X} \langle e_{x,y}, f\rangle.
$$
A.~M.~Vershik posed a question about combinatorial structure of polytopes 
$\KR(X)$ (equivalently, about combinatorial structure of $\LIP(X)$.)

For the special metric $\rho(x,y)=:\odin(x,y)=1$ at $x\ne y$, the
polytope $\KR(X)$ is called \emph{root polytope}, as its vertices 
are precisely the roots of the $A_n$ root system. 
This polytope is helpful in the study of the general case, cf. section \ref{unimod}.

We compute the $f$-vectors of this polytope $\KR(X)$ 
for a generic metric.  
Furthermore, we estimate from both sides 
number of types of metrics on the set ${\{1,\dots,n+1\}}$, classified by 
the combinatorial type of the (naturally labelled)
polytope $\KR$. 

\begin{opr}  We say that a
metric $\rho$ is \emph{generic}, if $\rho$ is strict 
and the polytope $\KR(X)$ is simplicial
(or, equivalently, the polytope $\LIP(X)$ is simple).
\end{opr}

Apparently, in the generic situation $f$-vector of the polytope $\LIP$
does not depend on the metric:

\begin{thm}\label{f}
Let $X$, $|X|=n+1$ be a metric space with generic metric $\rho$.
Then for $0\leqslant m\leqslant n$ the number of \((n-m)\)-dimensional faces 
of the polytope $\LIP(X)$ 
is equal to \(\binom{n+m}{m,m,n-m}=\frac{(n+m)!}{m!m!(n-m)!}\).
\end{thm}

Associate an oriented edge from $x$ to $y$ to a point $e_{x,y}$.  
Thereby
any face $\alpha$ (of arbitrary dimension) 
of the polytope $\KR(X)$ is associated to the graph
$D(\alpha)$ on the vertex set $X$ with edges corresponding to
signed measures $e_{x,y}$ lying on the face $\alpha$. 
Let $\tilde{D}(\alpha)$ denote the same graph $D(\alpha)$
with orientation forbidden.

\begin{opr} The collection of graphs of the form 
$D(\alpha)$ is called a \emph{combinatorial structure} of a dual pair 
of polytopes $\LIP(X)$, $\KR(X)$.
An oriented (unoriented, respectively) graph $G$ on the vertex set $X$ is called 
\emph{admissible}, 
if all its edges belong to some graph of the form $D(\alpha)$ ($\tilde{D}(\alpha)$, respectively.)
Two metrics $\rho_1,\rho_2$ on the same set $X$ are called
\emph{Lipschitz combinatorially  equivalent}, 
if the combinatorial structures of respective polytopes coincide. 
\end{opr}

The following theorem partially answers the question of A.~M.~Vershik \cite{V}. 
\begin{thm}\label{types} Suppose $|X|=n+1$.
The number $V(n)$, $V_g(n)$ of types of Lipschitz equivalence of, respectively, all metrics and generic metrics on the set $X$
satisfies the inequality
$$
c_1 n^2\leqslant \log V_g (n)\leq \log V(n)\leqslant c_2 n^3\log (n+1)
$$
for some positive absolute constants $c_1,c_2$.
\end{thm}

\section{Combinatorial description of faces}

The combinatorial properties of the
polytopes $\KR(X)$ 
have been considered before \cite{MPV,Z}, and some results
of this section have been obtained there. But these
results are not
enough for the purpose of this paper, this is why
we formulate and prove here
everything that we use.

A face $\alpha$ is the intersection 
of the polytope $\KR(X)$ and some support
hyperplane defined by the equation 
$\langle\mu,f_0\rangle=1$ for appropriate
function $f_0\in \LIP(X)$. The graph
 $D(\alpha)$  may be described in terms of  $f_0$ as follows:
edge from  $x$ to $y$ exists iff $f_0(x)-f_0(y)=\rho(x,y)$.

\begin{lemma}\label{dim} Dimension of a face $\alpha$ equals $n-c$, where
$c$ is the number of connected components of $\tilde{D}(\alpha)$. 
\end{lemma}

\begin{proof}
Assume that $X_1\subset X$. Then the signed measures
$e_{x,y}$, $x,y\in X_1$, generate a linear space of dimension $|X_1|-1$.
Next, if $G$  is a connected non-oriented graph
on the ground set $X_1$, the signed measures
 $e_{x,y}$ for $(x,y)\in E(G)$ generate this linear space 
(since $e_{x,z}$ is a linear combination of
$e_{x,y}$ and $e_{y,z}$). 
Thus the dimension of the linear span of points $e_{x,y}$, 
which belong to the face $\alpha$, equals
$n+1-c$, where $c$ is the number
of connected components in $\tilde{D}(\alpha)$.
Dimension of affine span is lesser by 1, since affine
span of a face does not contain the origin.
\end{proof}

Next theorem describes when a given set of signed measures
belongs to the same face of the polytope $\KR(X)$.

\begin{thm}\label{cyclic_thm} Let $G=(X,E)$
be an oriented graph on the ground set of vertices
$X$. Two conditions are equivalent:

(i) The graph $G$ is admissible, i.e., 
there exists a facet $\alpha$ of the polytope $\KR(X)$ such that
$E\subset D(\alpha)$;

(ii) there exists a function $f$, Lipschitz with constant
1, such that 
$f(x)-f(y)=\rho(x,y)$ whenever $(x,y)\in E$;

(iii) for any array of oriented edges $(x_i,y_i)\in E$, $i=1,\dots,k$,
inequality holds (here $y_{k+1}=y_1$):
\begin{equation}\label{cyclic}
\sum_{i=1}^k \rho(x_i,y_i) \leqslant \sum_{i=1}^k \rho(x_i,y_{i+1});
\end{equation}

(iv) Inequality \eqref{cyclic} holds
under additional assumption that all
points $x_1$, $x_2$,\dots, $x_k$ are distinct and all points  
$y_1$, $y_2$,\dots, $y_k$ are distinct. 
\end{thm}

\begin{proof} A support hyperplane
to the unit ball $\KR(X)$ of the space $F^*$ 
is given by the equation $\langle\mu,f\rangle=1$, where 
$f$ --- is a function from the unit sphere
$\LIP(X)$ of the space $F$. It yields that
(i) and (ii) equivalent. Clearly (iii) implies (iv). 
If (ii) holds, then  
\eqref{cyclic} follows from 
$$
\sum_{i=1}^k \rho(x_i,y_i)=\sum_{i=1}^k \left(f(x_i)-f(y_i)\right)=\sum_{i=1}^k \left(f(x_i)-f(y_{i+1})\right)\leqslant 
\sum_{i=1}^k \rho(x_i,y_{i+1}).
$$
It remains to prove (ii) assuming (iv). 

Existence of a necessary function $f$ may be rephrased as follows: the subspaces
\begin{gather*}
\{f\in F:f(y)-f(x)\leq \rho(x,y)\}, (x,y)\in X\times X;\\ 
\{f\in F:f(x)-f(y)\geq \rho(x,y)\}, (x,y)\in E,
\end{gather*}
must have non-empty intersection.
Since $F$ is $n$-dimensional, by Helly theorem
it suffices to prove that any subfamily of at most $n+1$ 
subspaces has non-empty intersection.
Assume the contrary and consider the least counterexample: at first,
by the number $n+1$ of points in
$X$, next, by the number 
$m\leq n+1$ of subspaces with empty intersection. 
Each subspace is defined by $f(y)-f(x)\leq \pm \rho(x,y)$, sign minus
is possible if $(x,y)\in E$. Call $x$ a starting point
and $y$ an endpoint. If some point $x$ does not serve neither as a starting
point nor as an endpoint for none of our subspaces, we have
a counterexample with $X\setminus \{x\}$ instead of $X$. The same holds
if  $x$ serves only as a staring point or only as an endpoint: 
a function from $X\setminus \{x\}$ may be extended to
 $x$ so that inequalities containing  $f(x)$ become true. Thus $m=n+1$ and each point
 $x\in X$ is an endpoint for exactly one subspace and a starting point
for exactly one subspace. A map, which send each starting point to the corresponding 
endpoint is a permutation of $X$. Let $z_1\dots z_sz_1$ be one of its cycles,
than $f$ should satisfy inequalities of the form 
$$
f(z_{i+1})-f(z_i)\leq \varepsilon_i \rho(z_i,z_{i+1}),i=1,\dots,s, \varepsilon_i\in \{-1,+1\}
$$
(as usual, agree that $z_{s+1}=z_1$). Such a function exists if and only if
$\sum \varepsilon_i \rho(z_i,z_{i+1})\geq 0$. Let $A$ 
be a set of indices  $i$ for which  $\varepsilon_i=-1$, $B$ 
be a set of other indices. Then $(z_i,z_{i+1})\in E$ for $i\in A$. Let $w(i)$, for $i\in A$, denote
the (first) index preceding 
$i$ in a cycle such that  $w(i)\in A$. Function  $w$ is a cyclic
permutation of  $A$. Condition (iv) yields that
$$
\sum_{i\in A} \rho(z_i,z_{i+1})\leqslant \sum_{i\in A} \rho(z_i,z_{w(i)+1})\leqslant \sum_{i\in B} \rho(z_i,z_{i+1})
$$
(last inequality follows from several triangle inequalities
which are summed up.). This is what we need.
\end{proof}

\begin{sled}\label{cor_cyclic}
1) All signed measures
 $e_{x,y}$ lie on the boundary of the polytope $\KR(X)$.

2) If a graph $G(\alpha)$ contains edges $(x,y)$, $(y,z)$, then 
$\rho(x,z)=\rho(x,y)+\rho(y,z)$. In particular, if $\rho$ is a strong metric,
then for any vertex of the graph $G(\alpha)$ either indegree or outderee equals 0.
In this case all signed measures $e_{x,y}$ 
are vertices of the polytope $\KR(X)$.

3) A strong metric $\rho$ is generic if and only if $\tilde{D}(\alpha)$ is a forest
for any face $\alpha$ of the polytope $\KR(X)$. 

4) Assume that a metric $\rho$ is strict and for
all mutually distinct points $x_1,\dots,x_k,y_1,\dots,y_k$, $y_{k+1}=x_1$, 
inequality
$$
\sum \rho(x_i,y_i)\ne \sum \rho(x_i,y_{i+1}),
$$
takes place. Then the metric $\rho$ is generic.

\end{sled}

\begin{proof} 
1) Graph with one edge satisfies (iv).

2) Set $x_1=x,x_2=y_1=y,y_2=z$ in  \eqref{cyclic}, we get 
$\rho(x,z)\geqslant \rho(x,y)+\rho(y,z)$, thus equality takes place.
If $\rho$ is a strict metric, then the function 
$f(z)=(\rho(y,z)-\rho(x,z))/2\in \LIP(X)$ satisfies equality
 $f(a)=f(b)+\rho(a,b)$ only for $a=x$, $b=y$. Thus corresponding
support plane has unique common point $e_{x,y}$ with the
ploytope $\KR(X)$. This yields that $e_{x,y}$ is a vertex of 
$\KR(X)$.

3) Strict metric $\rho$ is generic iff the polytope $\KR(X)$
is simplicial. That is, each face $\alpha$ of dimension
 $k=\dim \alpha$ contains exactly  $k+1$ 
vertices of  $\KR(X)$. Be Lemma \ref{dim} we have $k=n-c$, 
where $c$ is a number of connected components of $\tilde{D}(\alpha)$.
By p. 2) the number of vertices of $\KR(X)$ which belong to
 $\alpha$ equals the number of edges of the graph $\tilde{D}(\alpha)$. 
So, in terms of the graph $\tilde{D}(\alpha)$ condition is the following:
sum of the number  $k+1=n-c+1$ of edges and the number of
connected components  $c$ should be equal to the number of vertices
 $n+1$. Such graphs are exactly forests.

4) Consider a facet $\alpha$ of the polytope
$\KR(X)$. We have to check that 
 $\alpha$ has exactly $n$ vertices, i.e., that the graph
$D(\alpha)$ has exactly  $n$ edges. The graph
$\tilde{D}(\alpha)$ is connected by Lemma \ref{dim}. Thus
it has at  least $n$ edges. each vertex of the graph
 $D(\alpha)$ has indegree or outdegree  0, so, any cycle in 
$\tilde{D}(\alpha)$ is alternating: $y_1x_1y_2x_2\dots y_kx_k$,
$(x_i,y_i), (x_i,y_{i+1})\in E(D(\alpha))$.
Using \eqref{cyclic} twice we get 
$\sum \rho(x_i,y_i)\leqslant \sum \rho(x_i,y_{i+1})$ and
$\sum \rho(x_i,y_{i+1})\leqslant \sum \rho(x_i,y_i)$, thus equality
takes place. This contradicts to our assumption. Therefore there are no
cycles and the non-oriented graph $\tilde{D}(\alpha)$ is a tree,
it has exactly $n$ edges, as desired.
\end{proof}

Condition (iv) of Theorem \ref{cyclic_thm} may be weakened
in the case of facets. Namely, we have

\begin{thm}\label{tree-criterion}
Let $T$ be a tree on the ground set of vertices
 $X$. Orient $T$ so that indegree or outdegree of
each vertex is 0 (there are two ways to do it.) Obtained oriented
graph $T_d$ is contained in some $D(\alpha)$ for a certain
facet $\alpha$ of the polytope
$\KR(X)$ if and only if inequality
 \eqref{cyclic} takes place for any simple path
 $y_1x_1y_2x_2\dots y_{k}x_{k}$ in the (non-oriented) tree
$T$. 
\end{thm}

\begin{proof} Thie condition is necessary by Theorem \ref{cyclic_thm}. 
Let's prove that it is sufficient. There exists
the unique (up to additive constant) function
$f$ on $X$ such that $f(x)-f(y)=\rho(x,y)$ for each of
 $n$ directed edges $(x,y)$ of the graph $T_d$. 
Our goal is to prove that it satisfies $f(y)-f(x)\leq \rho(x,y)$
for all vertices $x,y\in X$. Induction in the length (
number of edges) of the path
$P$ from $y$ to
$x$ in $T$. 
For paths of length 1 it follows from the definition of   $f$.
Assume that we established this for paths shorter than between
 $x$ and $y$. Then
$f(x)=f(z)\pm \rho(x,z)$, where $z$ precedes $x$ in
 $P$. On the other hand, $f(y)-f(z)\leq \rho(y,z)$  by induction proposition. Thus
\begin{equation}\label{znaki}
f(y)-f(x)=f(y)-f(z)-(f(x)-f(z))\leq \rho(y,z)\mp \rho(x,z)
\end{equation}
Since 
$\rho(y,z)-\rho(x,z)\leq \rho(x,y)$ by triangle inequality, we get desired
inequality $f(y)-f(x)\leq \rho(x,y)$ if the sign in  \eqref{znaki} is negative.
It remains to consider the case $f(x)=f(z)-\rho(x,z)$. 
Analogously we may suppose that 
$f(y)=f(x_1)+\rho(y,x_1)$, where $x_1$ follows after $y$ in $P$. So, $P$ 
has even number of vertices, 
$P=y_1x_1\dots y_{k}x_{k}$, $y=y_1$, $x=x_{k}$,
$$
f(y)-f(x)=\sum_{i=1}^k \rho(x_i,y_i)-\sum_{i=1}^{k-1} \rho(x_i,y_{i+1})\leq \rho (x_{k},y_1)=\rho(x,y)
$$ due to \eqref{cyclic}. 
\end{proof}

Now we give a criterion that $\rho$ is generic.

\begin{thm}\label{generic_metric}
1) For all distinct points $x_1,\dots,x_k$ and distinct points
$y_1,\dots,y_k$ in  $X$ the set of directed edges $(x_i,y_i),1\leqslant i\leqslant k$
is admissible if and only if
\begin{equation}\label{gp_eq}
\sum_{i=1}^k \rho(x_i,y_i)=\min_{\pi} \sum_{i=1}^k \rho(x_i,y_{\pi(i)}),
\end{equation}
where minimum is taken over all permutations $\pi$ of the set $\{1,\dots,k\}$. 

2) A strict metric $\rho$ is generic
if and only if for any $2k$
distinct points $x_1,\dots,x_k$, $y_1,\dots,y_k$ in $X$ 
minimum of a sum in RHS of  \eqref{gp_eq}
is attained for unique permutation $\pi$.
\end{thm}

\begin{proof}
1) Let $C_1,C_2,\dots$ be non-empty disjoint subsets
of $\{1,\dots,k\}$ which are supports of the cycles of
$\pi$. If the set of edges
$(x_i,y_i)$ is admissible, then for each $C_j$ we have inequality
$$
\sum_{i\in C_j} \rho(x_i,y_i)\leqslant \sum_{i\in C_j} \rho(x_i,y_{\pi(i)})
$$
by \eqref{cyclic}. Summing up we get  
$$
\sum_{i=1}^k \rho(x_i,y_i)\leqslant \sum_{i=1}^k \rho(x_i,y_{\pi(i)}),
$$
this yields \eqref{gp_eq} as $\pi$ is arbitrary.
Conversely, if the set of edges $(x_i,y_i)$ is not admissible,
condition
(iv) of Theorem \ref{cyclic_thm} means that for a certain
cyclic permutation of some subset  $C\subset \{1,\dots,k\}$
the value of RHS of \eqref{gp_eq} is less than for the identical
permutation.

2) If a strong metric $\rho$ is not generic, than
some admissible graph $\tilde{D}(\alpha)$ contains a cycle.
By p.2 of Corollary \ref{cor_cyclic}, each vertex of $D(\alpha)$ 
has indegree 0 or outdegree 0, thus the vertices of this cycle
alternate and we may denote it
 $y_1x_1\dots y_kx_ky_1$, $(x_i,y_i),(x_i,y_{i+1})\in D(\alpha)$.
It means that minimum in \eqref{gp_eq} 
is obtained both for the identical permutation $\pi$ and for
$\pi(i)=i+1 \pmod k$. 

Assume now that $\rho$ is generic strict metric, but
minimum in  \eqref{gp_eq} is obtained for two different permutations.
Changing notations and considering a subset on which one of these two permutations
is a cyclic shift of another, we may consider the case when one of two
permutations is identical and another is a cyclic shift $\pi(i)=i+1 \pmod k$.
Let us show that a union of edges of the cycle $y_1x_1\dots y_kx_ky_1$ 
is admissible, by p.4 of Corollary \ref{cor_cyclic} 
this contradicts to our assumption that $\rho$ is generic.
We check condition $(iv)$. Choose few disjoint edges
from the cycle. They belong to some tree obtained from the cycle
by removing one edge. Thus it suffices to check that such a tree is
admissible. This may be checked by Theorem \ref{tree-criterion}. 
Indeed, condition of Theorem for any path in this tree
follows from the minimality of one or another permutation.
\end{proof}

\section{Stars and trees}

In this section we suppose that $\rho$ is a generic metric.

\begin{opr} A directed graph is called a \emph{star},
if there are edges coming from  one vertex to all the others,
and there are no other edges.
\emph{Constellation} is a directed graph in which all weak
connected 
components are stars.
\end{opr}

\begin{lemma}\label{ezhatnik}
Let \(V=\{v_1,\dots,v_k\}\subset X\) be a set consisting
of  $k$ points of a metric space $X$, \(p_1,\dots,p_k\) 
are non-negative integers and 
\[
k+\sum_{i=1}^k p_i = n+1.
\]
Then there exists unique admissible directed graph \(D^*\) such that
outdegrees of the vertices $v_i$ are equal to $p_i$ for $i=1,\dots,k$ 
(and their indegrees are equal to 0); indegrees of all vertices in 
$X\setminus V$ are equal to 1 (and indegrees are equal to 0). Also, the graph
\(D^*\) is a constellation and it minimizes the value of the functional
\[
F(D)=\sum_{(a,b)\in D} \rho(a,b)
\]
on the set of all graphs with described degrees.
\end{lemma}

\begin{proof} It is clear that any graph with described degrees is a constellation.

In order to prove existence we consider the constellation \(D^*\), which minimizes 
\(F\). Choose an array of edges as in p. (iv) of Theorem
 \ref{cyclic_thm}. The graph \(D^*\) does not contain edges of the form
 \((x_{i+1},y_i)\), since degrees of endpoints are equal to 1. Thus
 \(D^* \setminus \{(x_i,y_i)\} \cup \{(x_{i+1},y_i)\}\) is a constellation again,
and the value of \(F\) is not less than for \(D^*\). 
It yields condition (iv) of Theorem
\ref{cyclic_thm}, so, \(D^*\) is indeed admissible.

Now assume that there exists yet another
constellation \(D'\) with the same degrees. Choose the vertices
\(x_1\in V, y_1 \in X \setminus V\) so that \((x_1,y_1)\in D' \setminus D^*\). 
Since degree of \(y_1\) in both  \(D^*\) and \(D'\) equals 1, 
there exist  \(x_2\in V\) such that  \((x_2,y_1)\in D^* \setminus D'\). 
Next, degree of
\(x_2\) is the same in  \(D^*\) and \(D'\). Thus
there exists a vertex  \(y_2 \in X \setminus V\) such that
 \((x_2,y_2)\in D' \setminus D^*\). This process continues until
\(x_m=x_1\) for some \(m\) . We got a cycle (strictly speaking, disconnected orientation
of a non-directed cycle), all even edges of the cycle
belong to the first constellation and all odd edges to another. Since our two constellations 
are admissible, both sets
of odd and even edges minimize the sum in
\eqref{gp_eq}. It is impossible for a generic metric by Theorem \ref{generic_metric}.
\end{proof}

The following statement is straightforward. 

\begin{utv}\label{zanudnoe} Let $T$ 
be a tree on the ground set of vertices $X$, vertices
of $T$ are properly colored in black and white colors and $u\in X$ 
is a white vertex. Consider all edges $xy$ of the tree
 $T$ such that $x$ is white and the shortest path from $u$
to $x$ does not contain the edge $xy$. Such edges form
a constellation which we denote by $H(T,u)$. 
Denote by $\Phi_u(T)$ 
the sum 
$\sum \rho(x,y)$ of lengths of all edges in $H(T,u)$.
Let $P=y_1x_1y_2x_2\dots y_{k}x_{k}$ be a simple
path in the tree  $T$, in which the vertices $x_i$ are white
Denote by  $T'$ the tree which is obtained from $T$
by the change of the edge $x_ky_k$ to $x_ky_1$. Also denote by
$w$ the vertex of $P$ which is closest to $u$. Then
$$
\Phi_u(T')-\Phi_u(T)=
\begin{cases}
\rho(x_k,y_1)-\rho(y_1,x_1)+\rho(x_1,y_2)-\dots-\rho(y_k,x_k),&\text{if $w=x_k$}\\
0,&\text{else}.
\end{cases}
$$
\end{utv}

\begin{thm}\label{hyperface} 
Let $X=\{v_1,\dots,v_{n+1}\}$, 
nonengative integers $p_1,\dots,p_{n+1}$ sum up to 
$\sum p_i=n$. Then there exists unique
admissible graph such that outdegree of
$v_i$ equals $p_i$ for all $i=1,\dots,n+1$. 

If we color the vertices $v_i$ with 
$p_i>0$  in white and other vertices in black,
then this graph is a directed (from white to black) tree, and for each white
vertex  $u$ it minimizes the functional
$\Phi_{u}$ defined in the statement \ref{zanudnoe}.
\end{thm}

\begin{proof} P.4 of Corollary \ref{cor_cyclic} yields that any
admissible graph is a (somehow oriented) forest. Thus admissible graph
with $n$ edges is a tree.

Let us prove uniqueness of such a tree $T$. 
For any white vertex $u$ the constellation  $H(T,u)$ is admissible.
Degrees of white vertices in $H(T,u)$  depend on $u$, but not on $T$.
Lemma \ref{ezhatnik} implies that it is unique and minimizes  $\Phi_u$. 
The tree $T$ is a union of all such constellations $H(T,u)$, thus
it is at most unique.

It remains to prove the existence. Consider the tree
 $T$ with given outdegrees of white vertices,
for which the sum of functionals $\Phi_{u}$ ($u$ runs over white
vertices) is minimal possible. Let us claim that it is unique by verifying conditions
of Theorem \ref{tree-criterion}. By Claim \ref{zanudnoe},
if some path $y_1\dots x_{2k}$, where $x_i$ are white, obeys condition
of Theorem  \ref{tree-criterion}, then for the tree
$T'$ each functional $\Phi_{v_i}$ takes a value lesser or
equal than for $T$, and some functionals take strictly lesser value.
This contradicts to minimality assumption. 
\end{proof}

Theorem \ref{hyperface} implies theorem \ref{f} with $m=n$:

\begin{sled} The number of facets of the polytope $\KR(X)$ (or,
equivalently, the number of vertices of $\LIP(X)$) equals
 $\binom{2n}n$.
\end{sled}

\begin{proof} According to Lemma \ref{dim}, facets of $\KR(X)$ correspond
to admissible trees. By Theorem \ref{hyperface}
admissible trees are in bijective correspondence with 
sequences $(p_1,\dots,p_{n+1})$ of non-negative integers
which sum up to $n$. For any such sequence we may consider an increasing
sequence 
$(p_1+1,p_1+p_2+2,\dots,p_1+\dots+p_n+n)$ of numbers from
1 to  $2n$, thus their are exactly $\binom{2n}n$ of them. 
\end{proof}

\section{Rearrangements
}

The case $m=n$ of Theorem \ref{f} 
follows from Theorem \ref{hyperface}. Analogously,
the general case follows the following 

\begin{thm}\label{general_m} Let $X=\{v_1,\dots,v_{n+1}\}$, 
$p_1,\dots,p_{n+1}$ be non-negative integers
which sum up to  $m\leqslant n$. 
There exist exactly $\binom{n}{m}$ admissible graphs
such that outdegree of $v_i$ equals $p_i$ for  
all $i=1,2,\dots,n+1$.
\end{thm}

\begin{proof}

First of all, we prove Theorem \ref{general_m} in a special case.

Let \(X=\{1,\ldots,n+1\}\). Consider the metric
$$
\rho(i,j)=1+i/j,\, 1\leqslant i<j\leqslant n+1.
$$
For disjoint sequences $\{x_1<\dots<x_k\}\subset X$, 
$\{y_1<\dots <y_k\}\subset X$ the minimum in RHS of
 \eqref{gp_eq} is attained on the increasing permutation and only on it 
(this is known as ``rearrangement inequality'').
Thus the metric $\rho$ is generic and the graph
is admissible if and only if it does not contain
edges $(x_1,y_1), (x_2,y_2)$ for which $x_1<x_2$, $y_1>y_2$.
Assume that there are exactly $k$ positive numbers
among $p_1,\dots,p_{n+1}$, denote them
 $r_1,\dots,r_k$ in the order of corresponding points
on the real line. Denote by $Q$, $|Q|=n+1-k$, 
the set of other points of $X$. The statement of Theorem
\ref{general_m} in the case under consideration reduces to the following:

the number of subsets $A_1,\dots,A_k$ in $Q$ such that
$|A_i|=r_i$ for $i=1,\dots,k$ and $\max(A_i)\leqslant \min(A_{i+1})$ for $i=1,\dots,k-1$
equals $\binom{n}{m}=\binom{n}{\sum r_i}$.

This is clear. Indeed, without loss of generality we may suppose that $Q=\{1,2,\dots,n-k+1\}$.
Then the translates $A_1,A_2+1,A_3+2,\dots,A_k+(k-1)$
of the sets $A_1,\dots,A_k$ are disjoint and they form a set
of size $\sum r_i=m$ in $\{1,\dots,n\}$.

Now we start to deform our metric and control that the number
of admissible graphs with given outdegrees does not change.

Consider the metrics on $X$ as point of the phase space (of dimension $n(n+1)/2$:
$$
PS=\left\{f\colon X\times X \rightarrow \mathbb{R},\, f(x,y)=f(y,x),\, f(x,x)=0\,\text{for}\,x,y\in X\right\}.
$$
For any sequence $x_1,y_1,\dots,x_k,y_k$ 
of mutually distinct points in  $X$ we consider an
\emph{exceptional plane}
\begin{equation}\label{ploskost}
\sum_{i=1}^k f(x_i,y_i)=\sum_{i=1}^k f(x_i,y_{i+1}), \,\text{where}\,y_{k+1}:=y_1
\end{equation}
(some exceptional planes naturally coincide, we consider
only one in each class).

Consider two generic metrics $\rho_1,\rho_2$. Theorem \ref{tree-criterion} 
shows that while we change a metric continuously without meeting
exceptional planes, the set of admissible trees does not change. Theorem
 \ref{generic_metric} guarantees that the metric remains generic.

Replace each of the metrics $\rho_1,\rho_2$ to sufficiently close
and draw a segment between two new metrics. Almost
surely (in any reasonable sense, for example, with respect to
Lebesgue measure) this segment is not contained in no exceptional plane,
and it does not contain points which lie in at least two exceptional planes.
Thus we may suppose that when we move on this segment,
we meet at most one exceptional plane simultaneously. It remains
to prove that the number of admissible graphs from the statement of Theorem
\eqref{general_m} does not change after we intersect an
exceptional plane. Consider the moment of the 
intersection of the plane \eqref{ploskost}.
Assume that before intersection LHS of \eqref{ploskost} 
was less than RHS, and vice versa after intersection. 
Let's describe the rearrangement of the family of admissible graphs.
Admissible graphs which did not contain all $k$ edges
$(x_i,y_i)$ remain admissible (by property (iv) of Theorem \ref{cyclic_thm}). 
Graph $G$, which contains all these edges, is no longer admissible.
It corresponds to the following graph $G'$, which was not
admissible, but became admissible: \textit{for each index
$i$ such that $G$ did not contain the edge $(x_i,y_{i+1})$, add
it and remove $(x_i,y_i)$}. Note that outdegrees do not change after
such rearrangement. Let us prove that this graph $G'$ is admissible. Old graph
$G$ was contained in some admissible tree. It contained
all but one edges of the cycle $\gamma=y_1x_1\dots y_kx_ky_1$, else it would
remain admissible by Theorem \ref{tree-criterion}. 
This tree is changed by replacing one edge to another, and new tree $T'$ 
contains $G'$. Thus it suffices to check that $T'$ is admissible. 
Denote by $\rho$ the metric in the moment of rearrangement. 
There exists a function $f$ with Lipschitz constant 1 such that
$f(x)-f(y)=\rho(x,y)$ for all edges $(x,y)$ of the tree $T$ 
(condition (ii) of Theorem \ref{cyclic_thm} and passing to the limit). 
Thus the same equation holds for the only edge
of $T'\setminus T$ (i.e., for the edge of the cycle $\gamma$
which is absent in $T$: this follows from the
equation of the intersected plane). 
For other pairs of points we have a strict inequality
 $|f(x)-f(y)|<\rho(x,y)$. Thus the graph
$T\cup T'$ is admissible in the moment of rearrangement. Denote by $\rho'$
the metric after rearrangement. Consider a function
$f'$ such that $f'(x)-f'(y)=\rho'(x,y)$ for $(x,y)\in T'$. If
$(x,y)\notin T\cup T'$, then inequality $f(x)-f(y)<\rho(x,y)$ 
was strict, therefore it still holds for $f'$ and $\rho'$. For the the unique edge
of $T'\setminus T$ it also holds, since we intersected the plane.

So, we see that the number of admissible graphs with given outdegrees
does not decrease after we intersect the plane (the map $G\rightarrow G'$ is injective).
Analogously, it does not increase. Theorems \ref{general_m} and 
\ref{f} are proved.
\end{proof}

\section{Estimates of the number of types}\label{otsenki}

In this section we prove the estimates of Theorem \ref{types}.

In the previous section we considered exceptional planes. Here we need bit more
exceptional planes. Namely, consider also the planes determined by
not necessary distinct points $x_1,\dots,x_k,y_1,\dots,y_k$
(but $x$'s are distinct and $y$'s are distinct). Denote by
$N$ the number of such planes. They partition
the phase space $PS$ onto several parts (not necessary open,
for example, two points partition the line onto 5 parts: open interval,
two open rays and two points). Two functions
$f,g\in PS$ belong to the same part iff
$$
\sign I(f)=\sign I(g)
$$
for all linear functionals $I$, which determine exceptional planes.
Note that if two metrics belong to the same part, then
the families of admissible graphs for them coincide
by condition (ii) of Theorem \ref{cyclic_thm}. 
The graphs $D(\alpha)$, where  $\alpha$ is a facet of $\KR$, 
are inclusion-maximal admissible graphs.
Thus the families of facets for metrics $\rho_1,\rho_2$ coincide. 
Therefore the families of faces of lower dimension also coincide
(faces of lower dimensions are intersections of facets).
Thus the metrics $\rho_1,\rho_2$ are Lipschitz combinatorially  equivalent.
Therefore the number of types of combinatorial Lipschitz equivalence
does not exceed the number of parts defined by $N$ 
planes in the space of dimension $n(n+1)/2$. It is known that 
$N$ planes in the $d$-dimensional space determine at most
$$
2^d\binom{N}{d}+2^{d-1}\binom{N}{d-1}+\dots+\binom{N}0\leq \sum_{k=0}^d \frac{(2N)^k}{k!}\leq e(2N)^d
$$
parts. Since, obviously, $N\leq n^{2n}$, we get the upper
estimate in Theorem \ref{types}.

Now we come to the proof of the lower bound.
Fix a function $f\in PS$ such that its values
$f(x,y)$ for $x\ne y$ belong to the interval $(0,1)$ and are linearly
independent over  $\mathbb{Q}$. Consider 
$2^{n(n+1)/2}$ metrics of the type 
$\rho(x,y)=3\pm f(x,y)$ for $x\ne y$ (for all choices of signs.)
The claim is that at most $2^{o(n^2)}$ these metrics may be mutually
Lipschitz combinatorially  equivalent. It would immediately imply
the lower bound for generic metrics. Fix a metric $\rho_0$ and
estimate the number of metrics $\rho$ equivalent to
 $\rho_0$. Consider a graph on $X$, with edges corresponding to different signs
for $\rho_0$ and $\rho$. Assume that this graph
contains all edges $(x_1,y_1),(x_2,y_2),(x_1,y_2),(x_2,y_1)$
of a certain cycle of length 4. Apply
condition (iv) of Theorem \ref{cyclic_thm} for the edges $(x_1,y_1),(x_2,y_2)$.
We see that it holds
for exactly one of the metrics $\rho,\rho_0$. Therefore
$\rho,\rho_0$ are not Lipschitz combinatorially  equivalent.
Therefore the number of metrics $\rho$ which
are Lipschitz combinatorially  equivalent to $\rho_0$
does not exceed the number of graphs on $n+1$ vertices 
without 4-cycles. Such a graph contains at most $(n+1)^{3/2}$ edges (see,
for example, \cite{R}), which may be chosen by
at most $(n^2)^{(n+1)^{3/2}}=2^{o(n^2)}$ ways, as desired.

\section{Unimodular triangulations of the root polytope}\label{unimod}

Let $\rho$ be a metric on a set
$X$, $|X|=n+1$. Consider also the metric 
$\odin(x,y)=1,x\ne y$ on $X$. The vertices of the polytope
$\KR((X,\rho))$ lie on rays, which go from the origin
to the vertices of the polytope $\ROOT(X):=\KR((X,\odin))$. 
By Theorem \ref{cyclic_thm}, admissible graphs
for the metric $\odin$ are exactly all bipartite graphs
(with edges oriented from one part to another).  Therefore, if
the metric $\rho$ is strict, then the graph admissible for
$\rho$ is admissible also for $\odin$. If $\rho$ is also generic,
then for any facet of $\KR((X,\rho))$ we get a corresponding 
(under central projection) simplex belonging to
some facet of $\ROOT(X)$. Thus the central projection of the 
boundary of  $KR((X,\rho))$ onto the boundary of 
 $\ROOT(X)$ gives a triangulation of this last simplicial complex. 
Consider convex hulls of the simplices of this triangulation and the origin. 
We get a triangulation of the polytope $\ROOT(X)$ itself. 
Note two properties of these
triangulations. At first, they are 
\textit{regular}, in the sense that simplices of the triangulations are the linearity
set for the convex function: Kantorovich -- Rubinstein norm corresponding to
the metric
$\rho$. At second, they are \textit{unimodular}: 
all simplices in such a triangulation have equal volume. 
Indeed, any difference $\delta_x-\delta_y$
is expressed via analogous differences for any tree
as a linear combination with coefficients $\pm 1$; 
thus the linear maps which map simplices of 
our triangulation to each other
have integer coefficients, and their determinants 
are equal to $\pm 1$. It is known that all unimodular triangulations
of a lattice polytope have the same $f$-vector
(which may be defined invariantly via Ehrhart  polynomial, see
 \cite{S}). In turn, $f$-vectors of unimodular
triangulations of the root polytope were calculated in
 \cite{ABHP} (for concrete triangulation, as in the present paper),
and this gives another proof of Theorem
\ref{f}. However we remain a combinatorial proof,
which says more (Theorem \ref{general_m}).

From the other point of view, we may consider regular triangulations of
the polytope $\ROOT(X)$, which correspond to generic metrics, and
estimate the number of such triangulations from below as in
section \ref{otsenki}. Namely, fix a partition
$X=X_+\sqcup X_-$, $|X_+|=k$, $X_-=n+1-k$. It corresponds
to a bipartite oriented graph (edge go from $X_+$ to $X_-$). 
In turn, it corresponds to a facet  $\alpha_0$ of the polytope $\ROOT(X)$,
which is a product of simplices 
$\Delta^{k-1}\times \Delta^{n-k}$. 

Now we proceed as in section \ref{otsenki}. Fix a function
 $f$ on $X\times X$ such that its values $f(x,y)$
for $x\ne y$ belong to  $(0,1)$ and are linearly independent over
$\mathbb{Q}$. Consider all metrics of the type
$\rho(x,y)=3\pm f(x,y)$ for $x\ne y$ for all choices of signs.
For any such metric $\rho$ we get a polytope $\KR((X,\rho))$, 
for this polytope we get a regular triangulation of the corresponding facet
 $\alpha_0$ of the polytope $\ROOT(X)$. It is clear that it depends
only on the choice of signs for pairs $(x,y)$, $x\in X_+$, $y\in X_-$. Provide
an estimate for the number of choices of signs such that we get a metric
equivalent to a given metric $\rho_0$. Consider the bipartite graph
with parts $(X_+,X_-)$, in which the edges correspond to
different signs chosen for  $\rho_0$ and $\rho$.
Note that if this graph contains a  4-cycle, the pair of opposite
edges of this cycles is admissible for exactly one of two metrics $\rho,\rho_0$
(by condition
 (iv) of Theorem \ref{cyclic_thm}). Therefore the metrics
$\rho,\rho_0$ define different triangulations
of the facet $\alpha_0$. Therefore the number
of metrics  $\rho$ does not exceed the number of 4-cycle-free
spanning subgraphs of the complete bipartite
graph on  $(X_+,X_-)$. The number of edges
on such a graph does not exceed  $O(k(n-k+1)/\sqrt{n})$
(it follows from the standard argument that any pair
of vertices in one part has at most one common neighbor in another.) So we have proved

\begin{thm}\label{uni_simplex_otsenki}
Binary logarithm of the number of regular triangulations of the product
of simplices
$\Delta^{k-1}\times \Delta^{n-k}$ is at least
$$
k(n-k+1)-O(k(n-k+1)\cdot \log(n)\cdot n^{-1/2}).
$$
\end{thm}

For small $k$ this bound is worse than the known
bounds  \cite{Sa}.

\smallskip

We are grateful to A.~M.~Vershik for the attention and helpful
discussions.

\end{document}